\documentclass[10pt,oneside]{article}
\usepackage[top=3cm,bottom=3cm,left=2cm,right=2cm,headheight=1cm,headsep=1cm,footskip=1cm]{geometry}
\usepackage[toc,page]{appendix}
\usepackage{lipsum}
\usepackage[style,runin]{abstract}
\abslabeldelim{.}
\usepackage{titlesec}
\usepackage{cancel}
\usepackage{comment}
\usepackage[dvipsnames]{xcolor}
\usepackage{ragged2e}
\usepackage[english]{babel}
\usepackage[T1]{fontenc}
\usepackage{amssymb}
\usepackage{amsmath}
\usepackage{amsfonts}
\usepackage{amsthm}
\usepackage{accents}
\usepackage[usestackEOL]{stackengine}
\usepackage{scalerel}
\usepackage{tikz}
\usepackage{bbm}
\usepackage{mathrsfs}
\usepackage[none]{hyphenat}
\usepackage{graphicx}
\usepackage{float}
\usepackage{subfigure}
\usepackage{multicol}
\usepackage[most]{tcolorbox}
\usepackage{enumitem}
\usepackage{mathtools}
\usepackage{wrapfig}
\usepackage{url}
\usepackage{hyperref}
\usepackage[normalem]{ulem} 
\hypersetup{
  colorlinks   = true, 
  urlcolor     = blue, 
  linkcolor    = blue, 
  citecolor   = red 
}

\addto\mathspanish{}
\addto\mathspanish{}

\newcommand{\R}{\mathbb R}

\newcommand{\GG}{\mathcal G}

\newcommand{\N}{\mathbb N}

\newcommand{\Hi}{\mathcal H}

\newcommand{\RX}{\,]-\infty,+\infty]}

\newcommand{\RPP}{\ensuremath{\left]0,+\infty\right[}}

\newcommand{\RP}{\ensuremath{\left[0,+\infty\right[}}
\newcommand{\ifaf}{\:\Leftrightarrow\:}

\newcommand{\menge}[2]{\left\{{#1}~\middle|~{#2}\right\}} 
\newcommand{\scal}[2]{{\left\langle{#1}~\middle|~{#2}\right\rangle}}

\newcommand{\emp}{\ensuremath{{\varnothing}}}
\newcommand\norm[1]{\left\|#1\right\|}

\newcommand{\bigp}[1]{\left( #1 \right)}

\makeatletter
\newcommand{\ostar}{\mathbin{\mathpalette\make@circled\ast}}
\newcommand{\make@circled}[2]{%
  \ooalign{$\m@th#1\smallbigcirc{#1}$\cr\hidewidth$\m@th#1#2$\hidewidth\cr}%
}

\newcommand{\smallbigcirc}[1]{%
  \vcenter{\hbox{\scalebox{0.77778}{$\m@th#1\bigcirc$}}}%
}
\makeatother
\newcommand{\persp}[1]{\widetilde{#1}}
\def\overbar#1{\ThisStyle{%
  \setbox0=\hbox{$\SavedStyle#1$}%
  \stackengine{1.2\LMpt}{$\SavedStyle#1$}{\rule{\wd0}{.8\LMpt}}{O}{c}{F}{F}{S}%
}}
\newcommand{\cdom}{\overbar{\dom}\,}
\def\keywords#1{\par\addvspace\medskipamount{\rightskip=0pt 
plus1cm
\def\and{\ifhmode\unskip\nobreak\fi\ $\cdot$
}\noindent\keywordname\enspace\ignorespaces#1\par}}
\def\keywordname{{\bfseries Keywords}}
\def\acknowledgement{\par\addvspace{17pt}\small\rmfamily
\trivlist\if!\ackname!\item[]\else
\item[\hskip\labelsep
{\bfseries\ackname}]\fi}

\titleformat{\section}
  {\bfseries \Large}{\thesection}{1em}{}

\DeclareMathOperator{\dom}{dom}
\DeclareMathOperator{\prox}{prox}

\DeclareMathOperator{\Id}{Id}

\DeclareMathOperator*{\argmin}{arg\,min}
\DeclareMathOperator*{\rec}{rec}

\DeclareMathOperator{\inter}{int}
\DeclareMathOperator{\med}{mid}

\allowdisplaybreaks
\newtheorem{theorem}{Theorem}[section]
\newtheorem{proposition}{Proposition}[section]

\newtheorem{lemma}{Lemma}[section]
\theoremstyle{definition}
\newtheorem{definition}{Definition}[section]
\newtheorem{example}{Example}[section]
\newtheorem{remark}{Remark}[section]

\numberwithin{equation}{section}
\title{\Large \textbf{Enhanced computation of the proximity 
operator\\ 
for perspective functions}}
\author{Luis M. Brice\~{n}o-Arias$^1$, Crist\'obal 
Vivar-Vargas$^2$\\[.5cm]
\normalsize{$^1$Universidad T\'ecnica Federico Santa Mar\'ia, 
Departamento de Matem\'atica, Santiago, Chile} \\
\normalsize{\texttt{\href{mailto:luis.briceno@usm.cl}{luis.briceno@usm.cl}}}
 \\[2.5mm]
\normalsize{$^2$Universidad T\'ecnica Federico Santa Mar\'ia, 
Departamento de Matem\'atica, Santiago, Chile} \\
\normalsize{\texttt{\href{mailto:cristobal.vivar@usm.cl}{cristobal.vivar@usm.cl}}}}
\date{}
\begin{document}
\begin{titlepage}
	\centering
	\maketitle
	\vspace*{1cm}
	\begin{abstract}
In this paper we provide an explicit expression for the proximity 
operator of a perspective of any proper lower semicontinuous 
convex function defined on a Hilbert space. Our computation 
enhances and generalizes known formulae for the case when the 
Fenchel conjugate of the convex function has open domain or 
when it is radial.
We provide several examples of non-radial functions for which the 
domain of its conjugate is not open and we compute the proximity 
operators of their perspectives. 
\end{abstract}
\end{titlepage}
\section{Introduction}
\label{sec:1}
In this paper we enhance the computation of the proximity operator 
of the perspective of a proper lower semicontinuous convex 
function defined in a real Hilbert space $\Hi$.
This construction is introduced in \cite{rockafellar1966level} and 
appears naturally
in optimal mass transportation theory 
\cite{benamou2000computational,villani2021topics},
dynamical formulation of the 2-Wasserstein distance 
\cite{benamou2000computational,villani2021topics}, information 
theory \cite{elGheche2017proximity}, physics \cite{bercher2013}, 
operator theory \cite{effros2009}, statistics \cite{owen2007}, matrix 
analysis \cite{dacorogna2008},
signal processing and inverse problems 
\cite{kuroda2021convex,kuroda2022block,micchelli2013},
JKO \cite{jordan1998variational} schemes for gradient flows in the 
space of probability measures 
\cite{benamou2016augmented,carrillo2022primal}, and 
transportation and mean field games problems with 
state-dependent potentials
\cite{briceno2018proximal,cardaliaguet2015weak}, among other 
disciplines. 

The proximity operator of the perspective of a proper lower 
semicontinuous convex function $f\colon\Hi\to\RX$ is first obtained 
in \cite[Theorem~3.1]{combettes2018perspectiveprox} in the case 
when the domain of $f^*$ is open, where $f^*$ denotes the 
Fenchel-Legendre conjugate of $f$. Some examples in the case 
when the domain $f^*$ is closed are also explored in 
\cite[Section~3.2]{combettes2018perspectiveprox}. 
These results need the solution of an inclusion in $\Hi$ in order to 
compute the proximity operator of the perspective of $f$. In the 
case when $f$ is radial, the proximity operator of the perspective of 
$f$ is studied without any assumptions on the domain of its 
conjugate in \cite[Proposition 2.3]{combettes2020perspective}. 
This calculus needs a projection onto a particular convex subset of 
$\R^2$, which is not always easy to compute.
Previous results do not allow the computation of the proximity 
operator of the perspective of non-radial functions $f$ such that the 
domain of $f^*$ is not open. This is the case, for instance, of 
entropy-based penalizations including the log-sum, which arises, 
e.g.,  in optimal transport problems and mathematical programming 
\cite{auslender1997asymptotic,Pegon23}.

The goal of this paper is to provide a simple computation of the 
proximity operator of the perspective of any proper lower 
semicontinuous convex function without any further assumption. 
Our computation generalizes the results in 
\cite{combettes2018perspectiveprox} and 
\cite{combettes2020perspective} and the solution of a scalar 
equation is needed, which can be obtained via standard 
root-finding methods \cite{press2007numerical}. We provide 
several examples of non-radial functions $f$ such that the domain 
of $f^*$ is not open, for which the proximity operator is not 
available in the literature. 

The paper is organized as follows. In Section~\ref{sec:2} we 
present our notation and preliminary results. The proximity 
operator for the perspective of $f$ is presented in 
Section~\ref{sec:3}. Examples are studied in Section~\ref{sec:4}.

\section{Notation and preliminaries}
\label{sec:2}
\subsection{Notation}
Throughout this paper, $\mathcal{H}$ is a real Hilbert space 
endowed with the inner product $\scal{\cdot}{\cdot}$ and 
associated norm $\norm{\cdot}$. $\Hi \oplus \R$ denotes the 
Hilbert direct sum between $\Hi$ and $\R$. 

Let $f: \Hi \to \RX$. The domain of $f$ is $\dom f = \menge{x \in 
\Hi} {f(x) < +\infty}$ and $f$ is proper if $\dom f\neq \emp$. 
Denote by $\Gamma_0(\mathcal{H})$ the class of proper lower 
semicontinuous convex functions from 
$\Hi$ to $]-\infty, +\infty]$. 
Suppose that $f \in \Gamma_0(\Hi)$. 
The recession function of $f$ is
\begin{equation}
\label{eq:recession}
(\forall x_0 \in \dom f) (\forall x \in \Hi)\quad \rec f(x) = \lim_{t \to 
+\infty} \frac{f(x_0 + t x) - f(x_0)}{t}
\end{equation}
and it satisfies (see \cite[Proposition~7.13 \& 
Proposition~13.49]{bauschke2011convex})
\begin{equation}
\label{e:recession_char}
\rec{f} = \sigma_{\dom f^*} = \sigma_{\;\cdom f^*}.
\end{equation}
The Fenchel conjugate of $f$ is 
\begin{align}
\label{eq:conjugate}
f^*: \Hi \to \RX: u \mapsto \sup_{x \in \Hi} \bigp{\scal{x}{u} - f(x)}.
\end{align} 
We have $f^*\in\Gamma_0(\Hi)$, $f^{**} = f$, and we have the 
Fenchel-Young inequality 
\cite[Proposition~13.15]{bauschke2011convex}
\begin{align}
(\forall x \in \Hi)(\forall u \in \Hi) \quad f(x) + f^*(y) \geq 
\scal{x}{u}.\label{eq:fenchel_young}
\end{align}
The subdifferential of $f$ is the set-valued operator
\begin{align}
\label{eq:subdifferential}
\partial f: \Hi \to 2^{\Hi} : x \mapsto 
\menge{u \in \Hi}{(\forall y \in \Hi) \quad \scal{y-x}{u} + f(x) \leq f(y)}
\end{align}  
and $\dom \partial f=\menge{x\in\Hi}{\partial f(x)\neq\emp}$.
We have the
Fenchel-Young identity 
\cite[Proposition~16.10]{bauschke2011convex}
\begin{align}
(\forall x \in \Hi)(\forall u \in \Hi) \quad u \in \partial f(x) 
\quad\Leftrightarrow\quad f(x) + f^*(u) = \scal{x}{u}. 
\label{eq:fenchel_young_iden}
\end{align}
The \textit{proximity operator} of $f$ is
\begin{align}
\label{eq:prox_def}
\prox_f : \Hi \to \Hi : x \mapsto \argmin_{y \in \Hi} \left(f(y) + 
\frac{1}{2} \norm{x-y}^2\right),
\end{align}
which is characterized by
\begin{equation}
\label{eq:prox_char}
(\forall x \in \Hi)(\forall p \in \Hi) \quad p = \prox_f x 
\quad\Leftrightarrow\quad x-p \in \partial f(p)
\end{equation}
and satisfies 
\begin{align}
\label{eq:moreau_decomp}
(\forall\gamma \in\RPP)\quad  \prox_{\gamma f} =\Id - \gamma 
\prox_{f^*/\gamma}\circ (\Id/\gamma),
\end{align}
where $\Id\colon\Hi\to\Hi$ denotes the identity operator.

Let $C\subset\Hi$ be a nonempty closed convex set. The indicator 
function of $C$ is
\begin{equation}
 \label{eq:indicator}
\iota_C: \Hi \to \RX: x \mapsto 
\begin{cases}
0, &  \text{if }x \in C;\\
+\infty, &\text{if }x \not\in C,    
\end{cases}   
\end{equation}
its support function is
\begin{align}
\sigma_C : \Hi \to\RX: u \mapsto\sup_{x \in C} \scal{x}{u}, 
\label{eq:support}
\end{align}
and we have $\sigma_C= (\iota_C)^*$. The projection operator 
onto $C$ is $P_C=\prox_{\iota_C}$, which is characterized by
\begin{equation}
\label{eq:projection_char}
(\forall x \in \Hi)(\forall y \in C) \quad \scal{y - P_C x}{x - P_C x} \leq 
0.
\end{equation}
For further background on convex analysis, the reader is referred 
to \cite{bauschke2011convex}. 
\subsection{Perspective functions and properties}
Now, we review essential properties of perspective functions. We 
refer the reader to \cite{combettes2018perspective} for further 
background.
\begin{definition}\label{def:perspective} Let $f \in 
\Gamma_0(\Hi)$. The \textit{perspective} of $f$ is:
\begin{equation}
\widetilde{f}\colon\Hi\times\R\to\RX\colon
 (x,\eta) \mapsto 
\begin{cases}
\eta f \bigp{\dfrac{x}{\eta}},&\text{if}\:\:\eta > 0; \\[3mm]
(\rec f) (x), & \text{if}\:\:\eta = 0; \\[2mm]
+\infty, & \text{if}\:\:\eta < 0.
\end{cases}    
\label{eq:persp_def}
\end{equation}
\end{definition} 
\begin{lemma}
\label{lemma:persp_prop}
Let $f \in \Gamma_0(\Hi)$. Then the following hold:
\begin{enumerate}[label = \normalfont(\roman*)]
\item
\label{lemma:persp_propi}
$\persp{f} \in \Gamma_{0}(\Hi\oplus \R)$.
\item
\label{lemma:persp_propi+}
$\persp{f}$ is not radial, i.e., there exist $(x,\eta)$ and $(y,\nu)$ in 
$\Hi\times\R$ such that $\|(x,\eta)\|=\|(y,\nu)\|$ and 
$\persp{f}(x,\eta)\ne\persp{f}(y,\nu)$.
\item
\label{lemma:persp_propii}
Let $C = \menge{(x,\eta) \in \Hi \times \R}{\eta + f^*(x) \leq 0}$. 
Then $\bigp{\persp{f}}^* = \iota_C$. 
\end{enumerate}
\end{lemma}
\begin{proof}
\ref{lemma:persp_propi}: \cite[Proposition 
2.3(ii)]{combettes2018perspective}.

\ref{lemma:persp_propi+}: Let $(x,\eta) \in\Hi\times\RPP$ be such 
that $x/\eta \in \dom f$. Then $\norm{(x,\eta)} = \norm{(x,-\eta)}$ 
and $\persp{f}(x,\eta) = \eta f(x/\eta) < +\infty = \persp{f}(x,-\eta)$. 
Hence $f$ is not radial. 

\ref{lemma:persp_propii}: \cite[Proposition 
2.3(iv)]{combettes2018perspective}. \qed
\end{proof}

\subsection{Preliminaries on proximity operators}
We now provide some preliminary results 
needed in the following sections.
\begin{lemma}{\normalfont 
\cite[Lemma~3.1(iii)-(iv)]{bricenoarias2023proximity}}
\label{lemma:prox_fenchel}
Let $f \in \Gamma_0(\Hi)$, let $\gamma \in \RPP$, and let $(x,p) 
\in \Hi\times\Hi$. Then the following propositions are equivalent:
\begin{enumerate}[label = \normalfont(\roman*)]
\item $p = \prox_{\gamma f}x$.
\item $f(p) + f^*\left((x-p)/\gamma\right) = \scal{p}{(x-p)/\gamma}$.
\item $(\forall y \in \Hi) \quad \scal{x-p}{y-p} + \gamma f(p) \leq  
\gamma f(y)$.
\end{enumerate}
\end{lemma}

\begin{proposition}
\label{prop:prox_proj}
	Let $f \in \Gamma_0(\Hi)$, let $\gamma \in \RPP$, and let $x \in 
	\Hi$. Then the following hold:
\begin{enumerate}[label = \normalfont(\roman*)]
\item We have
\label{prop:prox_proji}
 \begin{equation}
    \label{eq:prox_proj}
	-\scal{ x-P_{\; \cdom f}\, x}{\prox_{\gamma f} x - P_{\; \cdom f}\, 
	x} + \norm{P_{\; \cdom f}\, x - \prox_{\gamma f} x}^2\leq \gamma 
	\bigp{f\bigp{P_{\; \cdom f}\, x} - f \bigp{\prox_{\gamma f} x}}
\end{equation}
\item
\label{prop:prox_projii}
$f \left( \prox_{\gamma f} x \right) \leq f \left( P_{\; \cdom f}\, x 
\right).$
\end{enumerate}	
 \end{proposition}
\begin{proof}
\ref{prop:prox_proji}:
	Let $x$ and $p$ in $\Hi$ be such that $p = \prox_{\gamma f} x$. 
	Then, by setting $y = P_{\; \cdom f}\, x$, it follows from 
	Lemma~\ref{lemma:prox_fenchel} that
	\begin{equation}
    \label{eq:prox_proj_proof_2}
	- \scal{x- P_{\; \cdom f}\, x }{p - P_{\; \cdom}\, x} + \norm{P_{\; 
	\cdom f}\, x - p}^2\leq \gamma \bigp{f \bigp{P_{\; \cdom f}\, x} - 
	f(p)}, 
	\end{equation}
	and \eqref{eq:prox_proj} follows.

\ref{prop:prox_projii}:
Since $\cdom f$ is nonempty, closed, convex, and 
$\prox_{\gamma f} x \in \dom \partial f \subset \cdom f$, it follows 
from \eqref{eq:projection_char} that
    \begin{equation}
    \label{eq:iprod_prox_proj}
        \scal{x- P_{\; \cdom f}\, x}{\prox_{\gamma f} x - P_{\; \cdom f}\, 
        x} \leq 0.
    \end{equation}
   Hence, the result follows from \ref{prop:prox_proji}.
\end{proof}

\section{Main results}
\label{sec:3}
The proximity operator of a perspective function when $\dom f^*$ 
is open is computed in \cite[Theorem 
3.1]{combettes2018perspectiveprox}. In the case of radial 
functions, this hypothesis is removed in \cite[Proposition 
2.3]{combettes2020perspective}. In this section we compute the 
proximity operator of $\persp{f}$ for any $f \in \Gamma_0(\Hi)$.

\begin{theorem} 
\label{teo:main_result} 
Let $f \in \Gamma_0(\Hi)$, let $\gamma \in \RPP$, and let $(x,\eta) 
\in \Hi \times \R$. Then the following hold:
    \begin{enumerate}[label = \normalfont(\roman*)]
\item 
\label{teo:main_result_i}
Suppose that $\eta + \gamma f^* \big(P_{\;\cdom f^*} \, 
(x/\gamma)\big) \leq 0$. Then
    \begin{equation}
    \label{eq:main_result_i}
        \prox_{\gamma \widetilde{f}}(x,\eta) = \bigp{x - \gamma P_{\; 
        \cdom f^*} \bigp{\frac{x}{\gamma}}, 0}. 
    \end{equation} 
\item 
\label{teo:main_result_ii}    
    Suppose that $\eta + \gamma f^* \big(P_{\;\cdom f^*} 
    \,(x/\gamma)\big) > 0$. Then there exists a unique 
    $\mu\in\big]0,\eta + \gamma f^* \big(P_{\;\cdom f^*} 
    \,(x/\gamma)\big)\big]$ such that
    \begin{equation}
    \label{eq:main_result_iimu}
        \mu = \eta + \gamma f^* 
        \bigp{\prox_{\frac{\mu}{\gamma}f^*}\bigp{\frac{x}{\gamma}}}. 
    \end{equation} 
Furthermore
    \begin{equation}
    \label{eq:main_result_ii}
        \prox_{\gamma \widetilde{f}}(x,\eta) = \bigp{x - \gamma 
        \prox_{\frac{\mu}{\gamma} f^{*}} \bigp{\frac{x}{\gamma}},\mu}.
    \end{equation}
    \end{enumerate}
\end{theorem}
\begin{proof}
    First note that 
    Lemma~\ref{lemma:persp_prop}\ref{lemma:persp_propi} asserts 
    that $\widetilde{f}\in\Gamma_0(\Hi\oplus \R)$. Let $(p,\mu) \in 
    \Hi \times \R$ be such that $(p, \mu) = \prox_{\gamma \persp f} 
    (x,\eta)$. It follows from Lemma \ref{lemma:persp_prop} and 
    Lemma \ref{lemma:prox_fenchel} that
    \begin{align}
    \label{e:auxt}
    (p,\mu) = \prox_{\gamma \persp{f}}(x,\eta)
    & \:\:\ifaf\:\:
    \persp{f}(p,\mu)  + \left(\persp{f}\right)^* \left(\frac{x-p}{\gamma}, 
    \frac{\eta-\mu}{\gamma}\right) = 
    \scal{(p,\mu)}{\left(\frac{x-p}{\gamma}, 
    \frac{\eta-\mu}{\gamma}\right)}\nonumber\\[2mm]
    & \:\:\ifaf\:\: \persp{f}(p,\mu)+ \iota_C \left(\frac{x-p}{\gamma}, 
    \frac{\eta-\mu}{\gamma}\right) = \scal{p}{\frac{x-p}{\gamma}} + 
    \mu \left(\frac{\eta-\mu}{\gamma}\right) \nonumber\\[2mm]
    &\:\:\ifaf\:\: \persp{f}(p,\mu) = \scal{p}{\frac{x-p}{\gamma}} + \mu 
    \left(\frac{\eta-\mu}{\gamma}\right) \text{ and }\:\: 
    \frac{\eta-\mu}{\gamma} + f^*\left(\frac{x-p}{\gamma}\right) \leq 0.
    \end{align}
    Moreover, since $(p, \mu) \in \dom\partial\persp{f}\subset \dom 
    \persp{f}$, we have $\mu \in \RP$. Then, let us consider two 
    cases.
\begin{enumerate}[label = \normalfont(\roman*)]
    \item 
    Suppose that $\mu = 0$. Then \eqref{e:auxt}, 
    \eqref{e:recession_char}, Lemma~\ref{lemma:prox_fenchel}, 
    and \eqref{eq:moreau_decomp} imply
    \begin{align}
    \label{eq:main_result_proof_mu0}
    (p, 0) = \prox_{\gamma \persp f} (x,\eta) 
    & \:\Leftrightarrow \: (\rec f)(p) = \scal{p}{\frac{x-p}{\gamma}} 
    \quad\text{and}\quad \frac{\eta}{\gamma} 
    +f^*\bigp{\frac{x-p}{\gamma}} \leq 0 \nonumber\\
    &\:\Leftrightarrow\: \sigma_{\; \cdom f^*}\, (p)+\iota_{\; \cdom 
    f^*}\, \bigp{\frac{x-p}{\gamma}} = 
    \scal{p}{\frac{x-p}{\gamma}}\quad\text{and}\quad 
    \frac{\eta}{\gamma} +f^*\bigp{\frac{x-p}{\gamma}} \leq 0 
    \nonumber\\
    & \:\Leftrightarrow\: p = \prox_{\gamma \sigma_{\, \cdom f^*}}x 
    \quad\text{and}\quad \frac{\eta}{\gamma} 
    +f^*\bigp{\frac{x-p}{\gamma}} \leq 0 \nonumber\\[1mm]
    & \:\Leftrightarrow\: p = x - \gamma P_{\; \cdom 
    f^*}\bigp{\frac{x}{\gamma}}\quad\text{and}\quad \eta + \gamma 
    f^* \bigp{P_{\; \cdom f^*} \bigp{\frac{x}{\gamma}}} \leq 0.
    \end{align}
    \item 
    Suppose that $\mu > 0$. Then it follows from 
    Lemma~\ref{lemma:prox_fenchel}, \eqref{eq:fenchel_young}, 
    and \eqref{eq:moreau_decomp} that
    \begin{align}
    \label{eq:main_result_proof_muP}
    (p, \mu) = \prox_{\gamma \persp f} (x,\eta) &\ifaf \mu 
    f\left(\frac{p}{\mu}\right) =\: \scal{p}{\frac{x-p}{\gamma}} + \mu 
    \left(\frac{\eta-\mu}{\gamma}\right) \nonumber \\
    &\hspace{18mm}\text{and} \:\: \frac{\eta-\mu}{\gamma} + 
    f^*\left(\frac{x-p}{\gamma}\right) \le 0 \nonumber \\[2mm]
    & \ifaf f\left( \frac{p}{\mu}\right) =\: 
    \scal{\frac{p}{\mu}}{\frac{x-p}{\gamma}} + 
    \frac{\eta-\mu}{\gamma} \nonumber \\
    &\hspace{18mm}\text{and}\:\: \frac{\eta-\mu}{\gamma} + 
    f^*\left(\frac{x-p}{\gamma}\right) \le 0 \nonumber \\[2mm]
    & \ifaf f\left(\frac{p}{\mu}\right) + f^*\left(\frac{x - 
    p}{\gamma}\right) = \scal{\frac{p}{\mu}}{\frac{x - p}{\gamma }} + 
    \frac{\eta-\mu}{\gamma} + f^*\left(\frac{x - 
    p}{\gamma}\right)\nonumber\\
    &\hspace{18mm}\text{and}\:\: \frac{\eta-\mu}{\gamma} + 
    f^*\left(\frac{x-p}{\gamma}\right) \leq 0 \nonumber\\
    & \ifaf f\left(\frac{p}{\mu}\right) +f^*\left(\frac{x - p}{\gamma 
    }\right)= \scal{\frac{p}{\mu}}{\frac{x -p}{\gamma}} \nonumber \\
    &\hspace{18mm}\text{and }\:\: \frac{\eta-\mu}{\gamma} = 
    -f^*\left(\frac{x-p}{\gamma}\right) \nonumber \\[2mm]
    &\ifaf \frac{p}{\mu} = \prox_{\frac{\gamma}{\mu} f} 
    \left(\frac{x}{\mu}\right)\:\: 
    \text{and} \:\: \mu = \eta + \gamma 
    f^*\left(\frac{x-p}{\gamma}\right) \nonumber \\
    &\ifaf p = x -\gamma \prox_{\frac{\mu}{\gamma} f^*} 
    \left(\frac{x}{\gamma}\right) \:\:\text{and}\:\: \mu = \eta + \gamma 
    f^*\left(\prox_{\frac{\mu}{\gamma} f^*} 
    \left(\frac{x}{\gamma}\right)\right).
    \end{align}
    Hence, Proposition~\ref{prop:prox_proj}\eqref{prop:prox_projii} 
    implies that
    \begin{equation}
    \label{eq:main_result_proof_ineq}
    0 < \mu  = \eta + \gamma f^* \bigp{\prox_{\frac{\mu}{\gamma} 
    f^*} \bigp{\frac{x}{\gamma}}} \leq \eta + \gamma f^* \bigp{P_{\; 
    \cdom f^*} \bigp{\frac{x}{\gamma}}}.
    \end{equation}
    \end{enumerate}  

    Altogether, if $\eta + \gamma f^* (P_{\; \cdom f^*}\, 
    (x/\gamma))\leq 0$ and supposing that $\mu > 0$, we arrive at a 
    contradiction with \eqref{eq:main_result_proof_ineq} and 
    therefore \eqref{eq:main_result_i} follows from 
    \eqref{eq:main_result_proof_mu0}. Conversely,  if $\eta + 
    \gamma f^* (P_{\; \cdom f^*}\, (x/\gamma)) > 0$ and supposing 
    that $\mu = 0$, we arrive at a contradiction with 
    \eqref{eq:main_result_proof_mu0} and therefore 
    \eqref{eq:main_result_ii} and \eqref{eq:main_result_iimu} follow 
    from \eqref{eq:main_result_proof_muP} and \cite[Lemma~3.2(ii) 
    \& Lemma~3.2(iii)]{bricenoarias2023proximity}.
    \end{proof}

In the case when $f^*$ has open domain, 
Theorem~\ref{teo:main_result} recovers 
\cite[Theorem~3.1]{combettes2018perspectiveprox}, as the 
following example illustrates.
\begin{example}\label{example:plc_opendom} Let $f \in 
\Gamma_0(\Hi)$, let $\gamma \in\RPP$, let $\eta \in \R$, and let 
$x \in \Hi$. Then the following hold:
\begin{enumerate}[label = \normalfont(\roman*)]
    \item
    \label{example:plc_opendom_i} 
    Suppose that $\eta + \gamma f^*\left(x/\gamma\right) \leq 0$. 
    Then $\prox_{\gamma \persp f} (x,\eta) = (0,0)$. 
    \item
    \label{example:plc_opendom_ii} 
    Suppose that $\dom f^*$ is open and that $\eta + \gamma 
    f^*\left(x/\gamma\right) > 0$. Then
    \begin{equation}
        \prox_{\gamma \persp f}(x,\eta) = (x-\gamma p, \eta + 
        \gamma f^*(p)), \label{eq:plc_opendom_prox}
    \end{equation}where $p$ is the unique solution to the inclusion
    \begin{equation}
    \label{eq:plc_opendom_inclusion}
        x \in \gamma p + (\eta + \gamma f^*(p)) \partial f^*(p). 
    \end{equation}
    If $f^*$ is differentiable at $p$, then $p$ is characterized by $y = 
    \gamma p + (\eta + \gamma f^*(p))\nabla f^*(p)$. 
\end{enumerate}
\end{example}
\begin{proof}
    \ref{example:plc_opendom_i}: Note that $\eta + \gamma f^* 
    (x/\gamma) \leq 0$ implies that $x/\gamma \in \dom f^*\subset 
    \cdom f^*$ and then $P_{\; \cdom f^*}\, (x/\gamma) = 
    x/\gamma$. Therefore, it follows from Theorem 
    \ref{teo:main_result}\ref{teo:main_result_i} that
    \begin{align}
        \prox_{\gamma \persp f} (x,\eta) = \left(x - \gamma P_{\,\cdom 
        f^*}\,\left(\frac{x}{\gamma}\right), 0\right) = \left(x - \gamma 
        \frac{x}{\gamma}, 0\right) = (0,0).\label{eq:plc_opendom_leq0}
    \end{align}
    \ref{example:plc_opendom_ii}: Suppose that $\eta + \gamma 
    f^*(P_{\, \cdom f^*}\, (x/\gamma)) \leq 0$. Then, since $\dom f^*$ 
    is open,
    $P_{\;\cdom f^*} (x/\gamma) \in \dom f^* = \inter(\dom f^*)$ and it 
    follows from Theorem \ref{teo:main_result}\ref{teo:main_result_i} 
    that $\prox_{\gamma \persp f} (x,\eta) = (x - \gamma P_{\,\cdom 
    f^*}\,\left(x/\gamma\right), 0)$. Hence, \eqref{e:auxt}, 
    \eqref{e:recession_char}, the fact that $\iota_{\cdom 
    f^*}(P_{\cdom f^*}(x/\gamma))=0$, and 
    \eqref{eq:fenchel_young_iden} yield
    \begin{align}
     \label{eq:plc_opendom_proxgeq0}
        \persp{f} \left(x-\gamma P_{\,\cdom 
        f^*}\left(\frac{x}{\gamma}\right),0\right) &= \rec f \left(x-\gamma 
        P_{\, \cdom f^*} \left(\frac{x}{\gamma}\right)\right) \nonumber \\
        & = \sigma_{\,\cdom f^*}\left(x-\gamma P_{\, \cdom f^*} 
        \left(\frac{x}{\gamma}\right)\right) \nonumber \\
        &= \scal{x-\gamma P_{\,\cdom 
        f^*}\left(\frac{x}{\gamma}\right)}{P_{\,\cdom 
        f^*}\left(\frac{x}{\gamma}\right)} .
    \end{align}
    Therefore, by \cite[Corollary~7.6(i)]{bauschke2011convex}, 
    $P_{\, \cdom f^*} (x/\gamma) \in \cdom f^* \setminus \inter (\dom 
    f^*)$ which is a contradiction.
    Therefore $\eta + \gamma f^*(P_{\, \cdom f^*}\, (x/\gamma)) >0$ 
    and it follows from Theorem 
    \ref{teo:main_result}\ref{teo:main_result_ii} that
    \begin{align}
        \prox_{\gamma \persp f} (x,\eta) = \left(x - \gamma 
        \prox_{\frac{\mu}{\gamma}f^*}\bigp{\frac{x}{\gamma}},\mu 
        \right), \label{eq:plc_opendom_resultproxgeq0} 
    \end{align} where $\mu$ is the solution of the equation
    \begin{align}
        \mu = \eta + \gamma f^* 
        \bigp{\prox_{\frac{\mu}{\gamma}f^*}\bigp{\frac{x}{\gamma}}}. 
        \label{eq:plc_opendom_resultmugeq0} 
    \end{align}Now, given $p \in \Hi$, it follows from 
    \eqref{eq:prox_char} that
    \begin{align}
    \label{eq:plc_opendom_geq0inclusion} 
        p = \prox_{\frac{\mu}{\gamma}f^*}\bigp{\frac{x}{\gamma}} 
        \:\Leftrightarrow\: x \in \gamma p + \mu \partial f^*(p) = 
        \gamma p + (\eta + \gamma f^*(p)) \partial f^*(p). 
    \end{align}
    Hence, \eqref{eq:plc_opendom_prox} and 
    \eqref{eq:plc_opendom_inclusion} follow from 
    \eqref{eq:plc_opendom_resultproxgeq0} and 
    \eqref{eq:plc_opendom_geq0inclusion}. Lastly, the claim when 
    $f^*$ is differentiable follows from $\partial f^*(p) = \{\nabla 
    f^*(p)\}$. \qed
\end{proof}
\begin{remark}
 Note that \cite[Theorem~3.1]{combettes2018perspectiveprox} 
 needs the solution to the inclusion 
 \eqref{eq:plc_opendom_inclusion} for computing $\prox_{\gamma 
 \persp f}$.
 By contrast, Theorem~\ref{teo:main_result} only needs to solve a 
 scalar equation, which can be obtained via standard root-finding 
 methods \cite{press2007numerical}.
\end{remark}
The next result illustrates Theorem~\ref{teo:main_result} in the 
particular case of radial functions.
\begin{proposition}
\label{pro:plc_radial} Let $\phi \in \Gamma_0(\R)$ be even, 
set $f = \phi \circ \norm{\cdot}$, let $\gamma \in \RPP$, and let 
$(x,\eta) \in \Hi \times \R$. Then the following hold:
    \begin{enumerate}[label = \normalfont(\roman*)]
\item 
\label{pro:plc_radial_i}
Suppose that $\eta + \gamma \phi^* \bigp{P_{\;\cdom \phi^*}\, 
(\norm{x}/\gamma)} \leq 0$. Then
    \begin{equation}
        \prox_{\gamma \widetilde{f}}(x,\eta) = 
        \begin{cases}
        \bigp{\bigp{1 - \gamma \dfrac{P_{\; \cdom \phi^*} 
        \bigp{\|x\|/\gamma}}{\|x\|}}x, 0},&\text{if}\:\:x\neq 0;\\[2mm]
        (0,0),&\text{if}\:\:x=0.
        \end{cases} \label{eq:plc_radial_leq0}
    \end{equation} 
\item 
\label{pro:plc_radial_ii}
    Suppose that $\eta + \gamma \phi^* \bigp{P_{\;\cdom \phi^*}\,  
    (\norm{x}/\gamma)} > 0$. Then there exists a unique 
    $\mu\in\RPP$ such that
\begin{equation}
\label{eq:plc_radial_geq0_mu}
\mu = \eta + \gamma \phi^* 
\bigp{\prox_{\frac{\mu}{\gamma}\phi^*}\bigp{\frac{\|x\|}{\gamma}}}. 
\end{equation} 
Furthermore
    \begin{equation}
\label{eq:plc_radial_geq0_prox}
        \prox_{\gamma \widetilde{f}}(x,\eta) = 
        \begin{cases}
            \bigp{\bigp{1 - \gamma \dfrac{\prox_{\frac{\mu}{\gamma} 
            \phi^{*}} \bigp{\|x\|/\gamma}}{\|x\|}}
            x ,\mu},&\text{if}\:\:x\neq 0;\\[2mm]
            (0,\eta+\gamma\phi^*(0)), &\text{if}\:\:x= 0.
        \end{cases}
    \end{equation}
    \end{enumerate}    
\end{proposition} 
\begin{proof}
First, since $\phi$ is even, \cite[Proposition 
13.21]{bauschke2011convex} implies that $\phi^*$ is even and 
\cite[Example 13.8]{bauschke2011convex} yields
\begin{equation}
    f^* = \phi^* \circ \|\cdot\| \label{eq:plc_radial_conjf}.
\end{equation}
Hence,
\begin{align}
\label{eq:plc_radial_inf}
    f(0) = \phi(0) = \inf_{x \in \R} \phi(x) = -\phi^*(0) = -f^*(0),
\end{align}
and it follows from $\phi^* \in \Gamma_0(\R)$ and 
\cite[Lemma~3.1(i) \& Lemma~3.2(ii)]{bricenoarias2023proximity} 
that
\begin{equation}
    (\forall x\in\Hi)\quad P_{\,\cdom f^*}(x)=
    \begin{cases}
        P_{\, \cdom \phi^*}(\norm{x}) 
        \dfrac{x}{\norm{x}},&\text{if}\:\: x\neq 0;\\
        0,&\text{if}\:\: x= 0.
    \end{cases} \label{eq:plc_radial_proj}
    \end{equation}
Therefore \eqref{eq:plc_radial_conjf}, \eqref{eq:plc_radial_inf}, and 
\eqref{eq:plc_radial_proj} yield
\begin{equation}
\label{eq:plc_radial_imgnorm}
    (\forall x \in \Hi) \quad f^*\bigp{P_{\, \cdom f^*} x} = \phi^* 
    \bigp{P_{\, \cdom \phi^*} (\norm{x})}.
\end{equation}
\ref{pro:plc_radial_i}: In this case \eqref{eq:plc_radial_imgnorm} 
yields
\begin{equation}
    \eta + \gamma f^* \bigp{P_{\; \cdom f^*} \bigp{\frac{x}{\gamma}}} 
    \leq 0.
    \label{eq:plc_radial_proof_i}
\end{equation} Therefore, the result follows from Theorem 
\ref{teo:main_result}\ref{teo:main_result_i} and 
\eqref{eq:plc_radial_proj}. \\[\baselineskip]
\ref{pro:plc_radial_ii}:  In this case \eqref{eq:plc_radial_imgnorm} 
yields
\begin{equation}
    \eta + \gamma f^* \bigp{P_{\, \cdom f^*} \bigp{\frac{x}{\gamma}}} 
    > 0. \label{eq:plc_radial_proof_ii}
\end{equation}Hence, it follows from Theorem 
\ref{teo:main_result}\ref{teo:main_result_ii} and \cite[Lemma 
3.2(iv)]{bricenoarias2023proximity} that there exists a unique $\mu 
\in \RPP$ such that
\begin{equation}
    \mu = \eta + \gamma f^*\bigp{\prox_{\frac{\mu}{\gamma} f^*} 
    \bigp{\frac{x}{\gamma}}} = \eta + \gamma 
    \phi^*\bigp{\prox_{\frac{\mu}{\gamma} 
    \phi^*}\bigp{\frac{\norm{x}}{\gamma}}},
    \label{eq:plc_radial_proof_ii_mu}
\end{equation}
and
\begin{equation}
\label{eq:plc_radial_proof_ii_prox}
        \prox_{\gamma \widetilde{f}}(x,\eta) = 
        \begin{cases}
            \bigp{\bigp{1 - \gamma \dfrac{\prox_{\frac{\mu}{\gamma} 
            \phi^{*}} \bigp{\|x\|/\gamma}}{\|x\|}}
            x ,\mu},&\text{if}\:\:x\neq 0;\\[2mm]
            (0,\mu), &\text{if}\:\:x= 0.
        \end{cases}
    \end{equation}
Furthermore, for $x=0$, \eqref{eq:plc_radial_proof_ii_mu} and 
\eqref{eq:plc_radial_proof_ii_prox} yield $\mu = \eta + \gamma 
\phi^* (\prox_{\mu \phi^*/\gamma }(0))$ and $\prox_{\gamma 
\persp{f}}(x,\eta) = (-\gamma \prox_{\mu f^*/\gamma} (0),\mu).$
Next, since $\phi(0) = -\phi^*(0)$ and $\mu \geq 0$, it follows from 
\eqref{eq:fenchel_young_iden} and \eqref{eq:prox_char} that
\begin{align}
\label{eq:plc_radial_proof_prox0}
    0 =\phi(0) + \phi^*(0)  \:\:\Leftrightarrow\:\: 0 \in 
    \frac{\mu}{\gamma} \partial \phi^*(0) \:\:\Leftrightarrow\:\: 0 = 
    \prox_{\frac{\mu}{\gamma}\phi^*}(0).
\end{align}
Therefore, the result follows from 
\eqref{eq:plc_radial_proof_ii_mu}, 
\eqref{eq:plc_radial_proof_ii_prox},
 and \eqref{eq:plc_radial_proof_prox0}. \qed
\end{proof}

\begin{remark}
\label{remark:radial} 
Note that \cite[Proposition~2.3]{combettes2020perspective} can be 
obtained from Proposition~\ref{pro:plc_radial}. Indeed, let us define 
\begin{align}
\label{eq:plc_radialorig_RS}
\mathcal{R} &= \menge{(\nu,\chi) \in \R^2}{\chi + \phi^*(\nu) \le 0} 
\nonumber\\
\text{and}\quad 
\mathcal{S} &= \menge{(\nu,\chi) \in \R^2}{\chi + 
\phi^*\left(P_{\,\cdom \phi^*}\nu\right) \le 0}
\end{align}
and note that $\mathcal{R}$ is nonempty,  closed, and convex, 
since it is the level set of the proper lower semicontinuous convex 
function $(\nu,\chi)\mapsto \chi + \phi^*(\nu)$.
Now let $(x,\eta)$ and $(\overline{\nu},\overline{\chi})$ in 
$\Hi\times\R$.
It follows from \eqref{eq:projection_char} that 
$(\overline{\nu},\overline{\chi})=P_{\mathcal{R}}(\|x\|/\gamma,\eta/\gamma)$
 if and only if
\begin{equation}
\label{eq:plc_radialorig_proj}
(\forall 
(\nu,\chi)\in\mathcal{R})\quad(\nu-\overline{\nu})\left(\frac{\|x\|}{\gamma}-\overline{\nu}\right)+
(\chi-\overline{\chi})\left(\frac{\eta}{\gamma}-\overline{\chi}\right)\le 
0.
\end{equation}
Moreover, since $\mathcal{R}\subset \mathcal{S}$, we have three 
cases:
\begin{enumerate}[label = \normalfont(\roman*)]
\item $(\|x\|/\gamma,\eta/\gamma)\in\mathcal{R}$: In this case we 
have $\norm{x}/\gamma \in \dom \phi^*$, and, then, $P_{\,\cdom 
\phi^*}(\norm{x}/\gamma) = \norm{x}/\gamma$. Hence, it follows 
from Proposition~\ref{pro:plc_radial}\ref{pro:plc_radial_i}
that
\begin{equation}
    \prox_{\gamma f}(x,\eta) = (0,0) \label{eq:plc_radialorig_i}.
\end{equation}

\item 
$(\|x\|/\gamma,\eta/\gamma)\in\mathcal{S}\setminus\mathcal{R}$: 
In this case,
note that $(\nu,\chi)\in\mathcal{R}$ implies 
$\nu\in\dom\phi^*\subset\cdom\phi^*$ and, hence, 
\eqref{eq:projection_char} yields
\begin{align}
\label{eq:plc_radialorig_iidem}
(\forall(\nu,\chi)\in\mathcal{R})\quad 
0&\ge\bigg(\nu-P_{\,\cdom\phi^*}\bigg(\frac{\|x\|}{\gamma}\bigg)\bigg)\bigg(\frac{\|x\|}{\gamma}-P_{\,\cdom\phi^*}\bigg(\frac{\|x\|}{\gamma}\bigg)\bigg)\nonumber\\
&=\bigg(\nu-P_{\,\cdom\phi^*}\bigg(\frac{\|x\|}{\gamma}\bigg)\bigg)\bigg(\frac{\|x\|}{\gamma}-P_{\,\cdom\phi^*}\bigg(\frac{\|x\|}{\gamma}\bigg)\bigg)
 \nonumber\\
&\hspace{45mm}+\Big(\chi-\frac{\eta}{\gamma}\Big)\Big(\frac{\eta}{\gamma}-\frac{\eta}{\gamma}\Big).
\end{align}
Therefore, \eqref{eq:plc_radialorig_proj} implies
$(\overline{\nu},\overline{\chi})=(P_{\,\cdom 
\phi^*}(\|x\|/\gamma),\eta/\gamma)=P_{\mathcal{R}}(\|x\|/\gamma,\eta/\gamma)$
and Proposition~\ref{pro:plc_radial}\ref{pro:plc_radial_i} yields
\begin{align}
\label{eq:plc_radialorig_iiprox}
\prox_{\gamma \widetilde{f}}(x,\eta)&=
\begin{cases}
\bigg(\Big(1-\frac{\gamma\overline{\nu}}{\|x\|}\Big)x,\eta-\gamma\overline{\chi}\bigg),&\text{if}\:\:
 x\neq 0;\\
(0,0),& \text{if}\:\: x=0.
\end{cases} \nonumber \\
&= \begin{cases}
\left(\left(1-\gamma\dfrac{ P_{\, \cdom \phi^*} 
(\norm{x}/\gamma)}{\|x\|}\right)x,0\right),&\text{if}\:\: x\neq 0;\\
(0,0),& \text{if}\:\: x=0.
\end{cases}
\end{align}

\item $(\|x\|/\gamma,\eta/\gamma)\in\R^2\setminus\mathcal{S}$: 
In this case, recalling that $\phi^*(0)=-\phi(0)$, we obtain from 
Proposition \ref{pro:plc_radial}\ref{pro:plc_radial_ii} that 
$\prox_{\gamma \widetilde{f}}(0,\eta)=(0,\eta-\gamma\phi(0))$.
On the other hand, set 
\begin{equation}
\label{eq:defmuchibar}
(\overline{\nu},\overline{\chi})=\left(\prox_{\frac{\mu}{\gamma} 
\phi^{*}} 
\bigp{\frac{\|x\|}{\gamma}},-\phi^*\left(\prox_{\frac{\mu}{\gamma} 
\phi^{*}} \bigp{\frac{\|x\|}{\gamma}}\right)\right),
\end{equation}
where $\mu\in\RPP$ is the unique solution to 
\eqref{eq:plc_radial_geq0_mu} guaranteed by 
Proposition~\ref{pro:plc_radial}\ref{pro:plc_radial_ii}.
It follows from Lemma~\ref{lemma:prox_fenchel} that 
\begin{equation}
\label{eq:proxsubd}
(\forall\nu\in\dom\phi^*)\quad  
(\nu-\overline{\nu})\left(\frac{\|x\|}{\gamma}-\overline{\nu}\right)\le 
\frac{\mu}{\gamma}(\phi^*(\nu)-\phi^*(\overline{\nu})).
\end{equation}
Now, let $(\nu,\chi)\in\mathcal{R}$ and recall that 
$\nu\in\dom\phi^*$. Hence, \eqref{eq:proxsubd}, 
\eqref{eq:plc_radial_geq0_mu}, and \eqref{eq:defmuchibar} yield
\begin{align}
\label{plc_radialorig_R2Sdem}
(\nu-\overline{\nu})\left(\frac{\|x\|}{\gamma}-\overline{\nu}\right)+
(\chi-\overline{\chi})\left(\frac{\eta}{\gamma}-\overline{\chi}\right)
&\le 
\frac{\mu}{\gamma}(\phi^*(\nu)-\phi^*(\overline{\nu}))+(\chi-\overline{\chi})\left(\frac{\eta}{\gamma}-\overline{\chi}\right)
 \nonumber\\
& = \left(\frac{\eta}{\gamma} + 
\phi^*(\overline{\nu})\right)(\phi^*(\nu)-\phi^*(\overline{\nu})) +(\chi+ 
\phi^*(\overline{\nu}))\left(\frac{\eta}{\gamma} + 
\phi^*(\overline{\nu})\right) \nonumber\\
& = \left(\frac{\eta}{\gamma} + \phi^*(\overline{\nu})\right) \left( 
\phi^*(\nu)+\chi\right) \nonumber \\
& = \frac{\mu}{\gamma}\left(\phi^*(\nu) + \chi\right) \nonumber\\ 
&\leq 0.
\end{align}
Therefore, 
$(\overline{\nu},\overline{\chi})=P_{\mathcal{R}}(\|x\|/\gamma,\eta/\gamma)$
and \cite[Proposition 2.3(iii)]{combettes2020perspective} is 
obtained from Proposition~\ref{pro:plc_radial}\ref{pro:plc_radial_ii}.
\end{enumerate}
Altogether, we deduce that \cite[Proposition 
2.3]{combettes2020perspective} is deduced from 
Proposition~\ref{pro:plc_radial}.
Note that, the formulae in 
\cite[Proposition~2.3]{combettes2020perspective} need the 
computation on $\mathcal{R}\subset\R^2$, which can be 
complicated in some instances, as the following example illustrates.
\end{remark}
\begin{example}
In the context of Proposition~\ref{pro:plc_radial}, let $\phi\colon 
x\mapsto x^2/2$.
Then, $\phi^*=\phi$, $\dom\phi^*=\R$, and, for every 
$\tau\in\RPP$, $\prox_{\tau \phi^*}=\Id/(1+\tau)$.
Therefore, given $(x,\eta)\in\Hi\times\R$, 
Proposition~\ref{pro:plc_radial} yields
\begin{equation}
\prox_{\gamma \widetilde{f}}(x,\eta)=
\begin{cases}
(0,0),&\text{if}\:\:\eta+\|x\|^2/(2\gamma)\le0;\\
(0,\eta),&\text{if}\:\:\eta+\|x\|^2/(2\gamma)>0\:\:\text{and}\:\:x=0;\\[1mm]
\left(\frac{\mu}{\gamma+\mu}x,\mu\right),&\text{if}\:\:\eta+\|x\|^2/(2\gamma)>0\:\:\text{and}\:\:x\ne0,
\end{cases}
\end{equation}
where $\mu\in\RPP$ is the unique solution to 
\begin{equation}
\label{eq:nonlin2}
\mu=\eta+\frac{\gamma}{2(\gamma+\mu)^2}\|x\|^2.
\end{equation}
On the other hand, the proximity operator of the perspective 
proposed in \cite[Proposition 2.3]{combettes2020perspective} 
needs the projection onto 
$$\mathcal{R}=\menge{(\nu,\chi)\in\R^2}{\chi+\nu^2/2\le 0},$$ 
which involves additional computations.
\end{example}
\section{Examples}
\label{sec:4}
In this section, we provide several instances in which $f$ is 
non-radial, $\dom f^*$ is not open, and 
Theorem~\ref{teo:main_result} allows us to compute the proximity 
operator of $f$. For this class of functions, the computation of the 
proximity operator of their perspectives is not available in the 
literature.

Let us start with the computation of the proximity operator of the 
perspective of the perspective of
a lower semicontinuous convex function.
\begin{example}
\label{example:perspectivex2} 
Let $\GG$ be a real Hilbert space, let $g \in \Gamma_0(\GG)$, and 
set $f = \persp{g}$. Then, Lemma 
~\ref{lemma:persp_prop}\ref{lemma:persp_propi} yields $f \in 
\Gamma_0(\GG \times \R)$ and, since $f$ is positively 
homogeneous 
\cite[Proposition~2.3(i)]{combettes2018perspective}, we have
\begin{equation}
    \persp{f}:((x,\eta),\delta) \mapsto \begin{cases}
        \eta g\left(\dfrac{x}{\eta}\right), 
        &\text{if}\:\:\eta>0\:\:\text{and}\:\:\delta\ge 0;\\[1.8mm]
        (\rec g)(x),&\text{if}\:\:\eta=0\:\:\text{and}\:\:\delta\ge 0;\\[1mm]
        +\infty, &\text{if}\:\:\eta<0\:\:\:\text{or}\:\:\:\delta < 0.
    \end{cases}
\end{equation}
Moreover, by defining $C = \menge{(x,\eta) \in \GG \times \R}{\eta 
+ g^*(x) \leq 0}$,
it follows from 
Lemma~\ref{lemma:persp_prop}\ref{lemma:persp_propii} that
\begin{equation}
\label{e:proxproj}
f^* = \iota_C\quad\text{and}\quad (\forall \tau\in\RPP)\quad 
\prox_{\tau f^*}=P_{\,\cdom f^*}=P_C.
\end{equation}
Hence, since $g^*\in\Gamma_0(\GG)$, $\dom f^* = C$ is closed.
Now, in order to compute the proximity operator of $\widetilde{f}$, 
fix $(x, \eta)\in \GG \times \R$, $\delta \in \R$, $\gamma \in \RPP$, 
and note that
\begin{align}
    \label{eq:perspersp_delta}
    \delta + \gamma f^*\bigp{P_{\, \cdom f^*} \bigp{\frac{x}{\gamma}, 
    \frac{\eta}{\gamma} }} = \delta + \gamma \iota_C\bigp{P_{C} 
    \bigp{\frac{x}{\gamma}, \frac{\eta}{\gamma} }} =  \delta. 
\end{align}
Therefore, by considering $\Hi=\GG\oplus\R$, we deduce from 
Lemma~\ref{lemma:persp_prop}\ref{lemma:persp_propi} and 
Theorem~\ref{teo:main_result}\ref{teo:main_result_i}
that, if $\delta\le 0$, $\prox_{\gamma\widetilde{f}}$ is computed in 
\eqref{eq:main_result_i}. On the other hand, if $\delta> 0$, 
Theorem~\ref{teo:main_result}\ref{teo:main_result_ii} asserts that
there exists a unique
$\mu \in\left]0,\delta\right]$  solution to $\mu = \delta + \gamma 
f^*(\prox_{\mu f^*/\gamma}(x/\gamma))$ and 
$\prox_{\gamma\widetilde{f}}$ is obtained in 
\eqref{eq:main_result_ii}. Altogether, noting that \eqref{e:proxproj} 
implies 
that $\mu =\delta$, we derive from \eqref{eq:moreau_decomp} and 
again from Theorem~\ref{teo:main_result} that
\begin{align}
\label{eq:prox_perspx2}
    \prox_{\gamma \persp{f}} ((x,\eta),\delta) &= 
    \bigp{(x,\eta) - \gamma P_{C}\bigp{\frac{x}{\gamma}, 
    \frac{\eta}{\gamma} }, \max\{0,\delta\}}\nonumber\\
    &=\bigp{\prox_{\gamma\widetilde{g}}(x,\eta), 
    \max\{0,\delta\}}\nonumber\\
    &=\begin{cases}
    \bigp{x-\gamma P_{\;\cdom 
    g^*}\left(\frac{x}{\gamma}\right),0,\max\{0,\delta\}},&\text{if}\:\:\ell(x,\gamma)\le
     0;\\[2mm]
    \bigp{x-\gamma\prox_{\frac{\nu}{\gamma} 
    g^*}\left(\frac{x}{\gamma}\right),\nu,\max\{0,\delta\}},&\text{if}\:\:\ell(x,\gamma)>0,\\
    \end{cases}
\end{align}
where $\nu\in \; ]0,\ell(x,\gamma)]$ is the unique solution to 
$\nu=\eta+\gamma g^*(\prox_{\nu  g^*/\gamma}(x/\gamma))$
and we denote
$\ell(x,\eta) = \eta+\gamma g^*(P_{\,\cdom g^*}(x/\gamma))$. 

In the particular case when $\GG = \R$ and $ g: \xi \mapsto 
\xi^2/2$, we have $ g =  g^*$, $\dom  g = \R$, and, for every $\tau 
\in \RPP$, $\prox_{\tau  g^*} = \Id /(1+ \tau)$. Therefore, 
\eqref{eq:prox_perspx2} reduces to
\begin{equation}
\prox_{\gamma f}((\xi,\eta),\delta) = \begin{cases}
    (0,0,\max\{0,\delta\}), & \text{if } \eta+\xi^2/(2\gamma) 
    \le0;\\[1mm]
    \left(\frac{\nu}{\gamma+\nu} \xi,\nu, \max\{0,\delta\}\right), & 
    \text{if } \eta+\xi^2/(2\gamma) > 0,
\end{cases}
\end{equation}
where $\nu \in \left]0,\eta+ \xi^2/(2\gamma)\right]$ is the unique 
solution to the cubic equation $\nu = \eta + \gamma 
\xi^2/(2(\gamma + \nu)^2)$.
\end{example}

The following three examples are motivated by penalty methods for 
solving convex-constrained mathematical programming problems 
investigated in \cite{auslender1997asymptotic}. The last example 
also appears in the dual of entropy-penalized transport problems 
\cite{Pegon23}.

\begin{example}
\label{example:nb_entropy}
Set
\begin{equation}
\psi\colon\xi\mapsto
\begin{cases}
-\ln(\xi), & \text{if}\:\:\xi>0; \\
+\infty, & \text{if}\:\:\xi\le 0,
\end{cases}
\end{equation}
set $\varphi=\psi+\iota_{]-\infty,1]}$, and set $f=\varphi^*$.
Since \cite[Example 13.2(iii)]{bauschke2011convex} implies that 
$\psi\in\Gamma_0(\R)$ and $]-\infty,1]$ is closed and convex,
we have $\varphi\in\Gamma_0(\R)$. Moreover, since 
$\dom\psi\cap ]-\infty,1[\neq\emp$, it follows from 
\cite[Theorem~15.3 \& Example~13.2(iii)]{bauschke2011convex} 
and simple computations that
$f\in\Gamma_0(\R)$ and
\begin{equation}
\label{eq:nburg}
    f:\R \to \RX \: : \: \xi \mapsto \begin{cases}
        -1-\ln(-\xi), &\text{if}\:\: \xi < -1; \\  
        \xi, & \text{if}\:\: \xi \ge -1,
    \end{cases}
\end{equation}
from which we obtain
\begin{equation}
    \persp{f}: \R \times \R \to \RX: (\xi,\eta) \mapsto \begin{cases}
        \eta - \eta \ln \left(-\frac{\xi}{\eta}\right), &\text{if}\:\: \eta > 0 
        \:\:\text{and}\:\: \xi < -\eta; \\
        \xi, &\text{if} \:\: \eta > 0 \:\:\text{and}\:\: \xi \ge -\eta; \\
        \max\{0,\xi\}, &\text{if}\:\: \eta = 0; \\
        +\infty, & \text{otherwise}.
    \end{cases}
\end{equation}
Moreover, note that 
\begin{equation}
\label{e:fetoil}
f^*=\varphi^{**}=\psi+\iota_{]-\infty,1]}\colon\xi\mapsto
\begin{cases}
-\ln(\xi), & \text{if}\:\:0<\xi\le 1; \\
+\infty, & \text{otherwise}
\end{cases}
\end{equation}
and, thus,
$\dom f^* = \left]0,1\right]$ which is neither open nor closed and 
$P_{\,\cdom f^*}: \xi \mapsto \med\{0,\xi,1\}$. 
Now, in order to compute the proximity operator of $\persp{f}$, fix 
$\xi \in \R$, $\eta \in \R$, $\gamma \in \RPP$, and note that
\begin{equation}
    \eta + \gamma f^*\left(P_{\,\cdom f^*} 
    \left(\frac{\xi}{\gamma}\right)\right) = \eta 
    -\gamma\ln\left(\med\left\{0,\frac{\xi}{\gamma},1\right\}\right).
\end{equation}
Therefore, by considering $\Hi = \R$, Theorem 
\ref{teo:main_result} yields
\begin{align}
\label{e:proxauzx}
    \prox_{\gamma \persp{f}} (\xi,\eta) = \begin{cases}
        \left(\xi- \gamma P_{\, \cdom f^*} \left(\frac{\xi}{\gamma}\right), 
        0\right), & \text{if}\:\: \eta 
        -\gamma\ln\left(\med\left\{0,\frac{\xi}{\gamma},1\right\}\right) 
        \leq 0; \\[3mm]
        \left(\xi - \gamma \prox_{\frac{\mu}{\gamma} 
        f^*}\left(\frac{\xi}{\gamma}\right), \mu\right), & \text{if}\:\: \eta 
        -\gamma\ln\left(\med\left\{0,\frac{\xi}{\gamma},1\right\}\right) > 
        0.
    \end{cases}
\end{align}
where $\mu \in ~]0, \eta -\gamma \ln 
\left(\med\left\{0,\xi/\gamma,1\right\}\right)[$ 
is the unique solution to 
\eqref{eq:main_result_iimu}.

Note that, it follows from \eqref{e:fetoil},
$]-\infty,1]\cap\dom\psi=]0,1]\neq\emp$, and 
\cite[Proposition~24.47]{bauschke2011convex} that
$\prox_{\mu f^*/\gamma}=P_{\,]-\infty,1]}\circ\prox_{\mu 
\psi/\gamma}$.
Moreover, \cite[Example~24.40]{bauschke2011convex}
implies that $\prox_{\mu 
\psi/\gamma}(\xi/\gamma)=(\xi+\sqrt{\xi^2+4\mu\gamma})/(2\gamma)$.
 
Observing that 
\begin{equation}
\dfrac{\xi+\sqrt{\xi^2+4\mu\gamma}}{2\gamma}\ge1\:\:\Leftrightarrow\:\:
 \xi\ge\gamma-\mu,
\end{equation}
we obtain 
\begin{equation}
\prox_{\frac{\mu}{\gamma}f^*}\Big(\frac{\xi}{\gamma}\Big)=\begin{cases}
        \dfrac{\xi+\sqrt{\xi^2+4\mu\gamma}}{2\gamma}, & \text{if } \xi 
        < \gamma-\mu;\\
        1, & \text{if } \xi\ge \gamma-\mu
    \end{cases}
\end{equation}
and \eqref{eq:main_result_iimu} reduces to
\begin{equation}
\mu=
\begin{cases}
\eta,&\text{if}\:\: \xi\ge \gamma-\eta;\\ 
\eta - \gamma \ln 
\bigg(\frac{\xi+\sqrt{\xi^2+4\mu\gamma}}{2\gamma}\bigg),&\text{if}\:\:
 \xi< \gamma-\eta.
\end{cases}
\end{equation}
Hence, since 
for all $(\xi,\eta) \in \Hi \times \R$, we have
\begin{equation}
\label{e:divis}
\eta -\gamma\ln\left(\med\left\{0,\frac{\xi}{\gamma},1\right\}\right) \le 
0\quad \Leftrightarrow\quad \eta \le 0\quad\text{and}\quad \xi \ge 
\gamma e^{\eta/\gamma}
\end{equation}
and 
\begin{align}
\eta -\gamma\ln\left(\med\left\{0,\frac{\xi}{\gamma},1\right\}\right) > 
0 \quad\Leftrightarrow\quad
&(\eta > 0 \:\:\text{and}\:\: \xi \ge \gamma - \eta) \nonumber\\
\text{or}\:\: &(\xi < \min \{\gamma e^{\eta/\gamma} , \gamma - 
\eta\}),
\end{align}
we deduce that \eqref{e:proxauzx} can be explicitly written as
\begin{align}
    \prox_{\gamma \persp{f}}(\xi,\eta) = \begin{cases}
        \left(\max\{0,\xi-\gamma\},0\right), & \text{if } \eta \le 0 \text{ 
        and } \xi \ge \gamma e^{\frac{\eta}{\gamma}}; \\
        (\xi-\gamma,\eta), & \text{if } \eta > 0 \text{ and } \xi \ge 
        \gamma -\eta; \\[2mm]
        \left(\dfrac{\xi-\sqrt{\xi^2+4\mu\gamma}}{2}, \mu\right), & 
        \text{if } \xi< \min\left\{\gamma e^{\frac{\eta}{\gamma}}, 
        \gamma-\eta \right\},
    \end{cases}
\end{align} where $\mu \in ~]0,\eta - \gamma 
\ln\left(\max\{0,\xi/\gamma\}\right)[$ is the unique solution to
\begin{equation}
\label{e:eqmuex}
\mu = \eta - \gamma \ln 
\bigg(\frac{\xi+\sqrt{\xi^2+4\mu\gamma}}{2\gamma}\bigg).
\end{equation}


\end{example}

\begin{example}
\label{example:nboltz_shan_entropy}
Let $n \in \N$ and consider
\begin{equation}
    f: \R^n \to \RPP: x \mapsto \sum_{i=1}^n e^{x_i-1}.
\end{equation}
Then, $f\in\Gamma_0(\R^n)$ and 
\begin{equation}
    \persp{f}: \R^n \times \R \to \RX: (x,\eta) \mapsto \begin{cases}
        \eta \sum_{i=1}^{n} e^{\frac{x_i}{\eta} - 1}, & \text{if}\:\:\eta > 0; 
        \\[1mm]
        0, & \text {if}\:\:\eta = 0\:\:\text{and}\:\: x\le 0;\\
        +\infty, &\text{otherwise}.
    \end{cases}
\end{equation}
In order to compute the proximity operator of $\persp{f}$, note that, 
in view of 
\cite[Example~13.2(v), Proposition~13.23 \& 
Proposition~13.30]{bauschke2011convex}, we obtain
\begin{equation}
    f^*:\R^n \to \RX: x \mapsto \sum_{i=1}^n \phi(x_i),
\end{equation}
where
\begin{align}
\label{eq:xlnx}
    \phi&: \R \to \RX :\xi \mapsto \begin{cases}
        \xi \ln (\xi), & \text{if}\:\:\xi >0; \\
        0 ,& \text{if}\:\: \xi = 0; \\
        +\infty, &\text{if}\:\:\xi < 0.
    \end{cases}
\end{align}
Hence, $\dom f^* = \RP^n$ is closed and $P_{\,\cdom f^*} = 
\max\{0,\cdot\}$, where we denote $\max\{0,\cdot\}\colon y \mapsto 
(\max\{0,y_i\})_{1\le i\le n}$. Now fix $(x,\eta)\in\R^n\times\R$, 
$\gamma\in\RPP$, and note that
\begin{equation}
\eta+\gamma f^*\left(P_{\,\cdom 
f^*}\left(\frac{x}{\gamma}\right)\right)=\eta + \sum\limits_{i\in I_+(x)} 
x_i\ln\left(\frac{x_i}{\gamma}\right),
\end{equation}
where
$I_+(x)=\menge{i\in\{1,\ldots,n\}}{x_i>0}$ and the sum over the 
empty set is zero.
Therefore, by setting $\Hi=\R^n$,
Theorem~\ref{teo:main_result} yields
\begin{align}
\label{eq:prox_nboltz_shan}
    \prox_{\gamma \persp{f}}(x,\eta) & = \begin{cases}
        \left(x - \gamma \max\left\{0,\dfrac{x}{\gamma}\right\}, 0 \right), 
        & \text{if}\:\: \eta + \sum\limits_{i\in I_+(x)} 
        x_i\ln\left(\dfrac{x_i}{\gamma}\right) \leq 0; \\[2mm]
        \left(x-\gamma \prox_{\frac{\mu}{\gamma} f^*} 
        \left(\dfrac{x}{\gamma}\right),\mu\right), & \text{if}\:\: \eta 
        +\sum\limits_{i\in I_+(x)} x_i\ln\left(\dfrac{x_i}{\gamma}\right) > 
        0,
    \end{cases}
\end{align}
where $\mu \in \; ]0,\eta + \sum_{i\in I_+(x)} x_i\ln(x_i/\gamma)]$ is 
the unique solution to \eqref{eq:main_result_iimu}.

Furthermore, in view of \cite[Proposition 
16.9]{bauschke2011convex}, we have $\prox_{\mu f^* /\gamma} 
(x/\gamma) \in \dom \partial(\mu f^*/\gamma) = \RPP^n$. Hence, 
we obtain from \eqref{eq:prox_char} that, for every $p\in\RPP^n$,
\begin{align}
p = 
\prox_{\frac{\mu}{\gamma}f^*}\left(\frac{x}{\gamma}\right)\quad&\ifaf\quad
 \frac{x}{\gamma}-p \in \frac{\mu}{\gamma} \partial f^*(p)\nonumber\\
&\ifaf\quad  (\forall i \in \{1,...,n\}) \quad \frac{x_i}{\gamma}-p_i \in 
\frac{\mu}{\gamma} \partial \phi(p_i) \nonumber \\
& \ifaf\quad  (\forall i \in \{1,...,n\})\quad\frac{x_i}{\mu}- 1 = \ln(p_i) + 
\frac{\gamma}{\mu}p_i \nonumber \\
& \ifaf\quad  (\forall i \in 
\{1,...,n\})\quad\frac{\gamma}{\mu}e^{\frac{x_i}{\mu}-1} = 
\frac{\gamma}{\mu} p_i \, e^{\frac{\gamma}{\mu} p_i} \nonumber \\
& \ifaf\quad  (\forall i \in \{1,...,n\})\quad p_i = \frac{\mu}{\gamma} 
W_0\left(\frac{\gamma}{\mu} e^{\frac{x_i}{\mu}-1}\right),
\end{align}
where $W_0$ is the principal branch of the Lambert W-function. 
Altogether,
by denoting $\min\{0,\cdot\}\colon y\mapsto -\max\{0,-y\}$ and, for 
every $y\in\R^n$, $e^y=(e^{y_i})_{1\le i\le n}$ and 
$W_0(y)=(W_0(y_i))_{1\le i\le n}$,
\eqref{eq:prox_nboltz_shan} reduces to
\begin{align}
\label{eq:prox_nboltz_shan2}
    \prox_{\gamma \persp{f}}(x,\eta) & = \begin{cases}
        \left(\min\left\{0,x\right\}, 0 \right), & \text{if } \eta + 
        \sum\limits_{i\in I_+(x)} x_i\ln(x_i/\gamma) \leq 0; \\[2mm]
        \left(x-\mu W_0\left(\frac{\gamma}{\mu} 
        e^{\frac{x}{\mu}-1}\right),\mu\right), & \text{if } \eta 
        +\sum\limits_{i\in I_+(x)} x_i\ln(x_i/\gamma) > 0,
    \end{cases}
\end{align}
where $\mu$ is the unique solution to 
\eqref{eq:main_result_iimu}, which reduces to
\begin{align}
\label{eq:Wlambert_mu}
    \mu &= \eta + \mu \sum_{i=1}^n W_0\left(\frac{\gamma}{\mu} 
    e^{\frac{x_i}{\mu}-1}\right) \ln 
    \left(\frac{\mu}{\gamma}W_0\left(\frac{\gamma}{\mu} 
    e^{\frac{x_i}{\mu}-1}\right)\right).
\end{align}
Since $W_0$ is continuous and strictly increasing in $\RP$ 
\cite{Corl96}, $\mu$ can be computed via standard 
one-dimensional root finding numerical schemes \cite[Chapter 
9]{press2007numerical}.

\end{example}

\begin{example}
Let $n \in \N$ and let
\begin{equation}
    f: \R^n \to \R : x \mapsto \ln\left(\sum_{i=1}^n e^{x_i}\right).
\end{equation} 
Then $f \in \Gamma_0(\R^n)$ and
\begin{equation}
    \persp{f}: \R^n \times \R \to \RX: (x,\eta) \mapsto \begin{cases}
        \eta \ln\left(\sum_{i=1}^n e^{\frac{x_i}{\eta}}\right), & \text{if } 
        \eta > 0; \\
        \max_{1\le i\le n}x_i, & \text{if } \eta =0; \\
        +\infty,& \text{if } \eta < 0,
    \end{cases}
\end{equation}
which appears naturally in the dual of entropy-penalized transport 
problems \cite{Pegon23}.

In order to compute the proximity operator of $\persp{f}$, set $\phi$ 
as in \eqref{eq:xlnx} and let
\begin{equation}
    \triangle = \menge{x \in \RP^n}{\sum_{i=1}^n x_i = 1}
\end{equation} be the probability simplex in $\R^n$. Then 
\cite[Example 3.25]{boyd2004convex} yields
\begin{equation}
    f^*:\: \R^n \to \RX: x \mapsto \begin{cases}
        \sum_{i=1}^n \phi(x_i), &\text{if} \:\: \sum_{i=1}^n x_i = 1;\\
        +\infty, & \text{otherwise,}
    \end{cases}
\end{equation}
and hence, $\dom f^* = \triangle$ is closed. Now, fix $x \in \R^n$, 
$\eta \in \R$, and $\gamma \in \RPP$. By setting $\Hi = \R^n$,
Theorem \ref{teo:main_result} implies
\begin{align}
\label{eq:prox_logsum}
    \prox_{\gamma \persp{f}}(x,\eta) =\begin{cases}
        \left(x-\gamma P_{\triangle}\left(\frac{x}{\gamma}\right), 
        0\right), & \text{if}\:\: \eta + \gamma 
        f^*\left(P_{\triangle}\left(\frac{x}{\gamma}\right)\right) \le 0, 
        \\[2mm]
        \left(x-\gamma \prox_{\frac{\mu}{\gamma}f^*} 
        \left(\frac{x}{\gamma}\right),\mu\right), & \text{if}\:\: \eta + 
        \gamma f^*\left(P_{\triangle}\left(\frac{x}{\gamma}\right)\right) 
        > 0,
    \end{cases}
\end{align}
where $\mu \in\left]0,\eta + \gamma 
f^*\left(P_{\triangle}(x/\gamma)\right)\right]$ is the unique solution 
to \eqref{eq:main_result_iimu}.
Moreover, since \cite[Proposition 16.9]{bauschke2011convex} yield 
$\prox_{\mu f^*/\gamma}(x/\gamma)\in\dom \partial(\mu f^*/ 
\gamma) = \triangle\cap\RPP^n$, we obtain from 
\eqref{eq:prox_char} that, for every $p \in \triangle\cap\RPP^n$,
\begin{align}
\label{eq:logsum_Wlambert}
    p = \prox_{\frac{\mu}{\gamma} f^*}\left(\frac{x}{\gamma}\right) & 
    \ifaf\frac{x}{\gamma} - p \in \frac{\mu}{\gamma}\partial f^*(p) 
    \nonumber \\
    &\ifaf
    (\exists \lambda \in \R)(\forall i \in \{1,...,n\}) \quad \frac{x_i - 
    \gamma p_i}{\mu} = \ln(p_i) + 1 + \lambda \nonumber\\
    & \ifaf (\exists \lambda \in \R)(\forall i \in \{1,...,n\}) \quad 
    \frac{x_i}{\mu} -1-\lambda = \ln(p_i) +\frac{\gamma}{\mu} 
    p_i\nonumber\\[1mm]
    & \ifaf (\exists \lambda \in \R)(\forall i \in \{1,...,n\}) \quad 
    \frac{\gamma}{\mu} e^{\frac{x_i}{\mu} -1-\lambda} = 
    \frac{\gamma}{\mu}p_i\,e^{\frac{\gamma}{\mu}p_i} 
    \nonumber\\[1mm]
    & \ifaf (\exists \lambda \in \R)(\forall i \in \{1,...,n\})\quad p_i =  
    \frac{\mu}{\gamma}W_0\left(\frac{\gamma}{\mu} 
    e^{\frac{x_i}{\mu} -1-\lambda}\right),
\end{align}
where $W_0$ is the principal branch of the Lambert W-function. 
Note that by summing over $i$ in \eqref{eq:logsum_Wlambert} we 
obtain
\begin{align}
    \frac{\mu}{\gamma}\sum_{i=1}^n W_0 \left(\frac{\gamma}{\mu} 
    e^{\frac{x_i}{\mu} -1-\lambda}\right) = 1.
\end{align}
Altogether, \eqref{eq:prox_logsum} reduces to
\begin{align}
\label{eq:prox_logsumfinal}
    \prox_{\gamma \persp{f}}(x,\eta) =\begin{cases}
        \left(x-\gamma P_{\triangle}\left(\frac{x}{\gamma}\right), 
        0\right), & \text{if}\:\: \eta + \gamma 
        f^*\left(P_{\triangle}\left(\frac{x}{\gamma}\right)\right) \le 0, 
        \\[2mm]
        \left(x-\mu W_0\left(\frac{\gamma}{\mu} e^{\frac{x}{\mu} - 
        \Lambda}\right) ,\mu\right), & \text{if}\:\: \eta + \gamma 
        f^*\!\left(\!P_{\triangle}\left(\frac{x}{\gamma}\right)\right) > 0,
    \end{cases}
\end{align} where we denote $\Lambda = (\lambda + 1)_{1 \leq 
i\leq n}$, and $\mu$ and $\lambda$ are the solution to the 
nonlinear system of equations
\begin{align}
\label{eq:logsum_system}
    \begin{cases}
    \mu &= \eta + \mu \sum_{i=1}^n W_0\left(\frac{\gamma}{\mu} 
    e^{\frac{x_i}{\mu} -1-\lambda}\right) 
    \ln\left(\frac{\mu}{\gamma}W_0\left(\frac{\gamma}{\mu} 
    e^{\frac{x_i}{\mu} -1-\lambda}\right)\right); \\[2mm]
    1 & = \frac{\mu}{\gamma}\sum_{i=1}^n W_0 
    \left(\frac{\gamma}{\mu} e^{\frac{x_i}{\mu} -1-\lambda}\right),
    \end{cases}
\end{align}
which can be solved, for instance, by Newton's type methods.
\end{example}

\section{Conclusions}
In summary, we provide an explicit formula for the proximity 
operator of the perspective function of any proper lower 
semicontinuous convex function defined in real Hilbert spaces. The 
formula needs to solve a scalar nonlinear equation which can be 
efficiently solved by several one-dimensional root-finding numerical 
schemes. Our result generalizes 
\cite{combettes2018perspectiveprox} and 
\cite{combettes2020perspective}, valid only under additional 
assumptions. Finally, we derive several new formulae of proximity 
operators of perspective functions arising in penalization of 
mathematical programming problems appearing, e.g., in 
entropy-penalized optimal transport problems.

\textbf{Acknowledgments.} The work of Luis M. Brice\~no-Arias is 
supported by  Centro de Modelamiento Matem\'atico (CMM), 
FB210005, BASAL fund for centers of excellence, FONDECYT 
1190871, and FONDECYT 1230257 from ANID-Chile.

\end{document}